\documentclass[12pt]{amsart}
\usepackage[margin=1.15in]{geometry}
\usepackage{graphicx}
\usepackage{graphics}
\usepackage{amssymb}
\usepackage{enumerate}
\usepackage{epstopdf}
\usepackage{amsmath,amsthm,amscd,amssymb}
\usepackage{latexsym}
\usepackage[colorlinks,citecolor=red,pagebackref,hypertexnames=false]{hyperref}
\numberwithin{equation}{section}
\newtheorem{theorem}{Theorem}[section]
\newtheorem{corollary}[theorem]{Corollary}

\newtheorem{proposition}[theorem]{Proposition}

\newtheorem{example}[theorem]{Example}
\newtheorem{definition}[theorem]{Definition}

\begin{document}
	
	\title{Random walks and the "Euclidean" association scheme in finite vector spaces}
	
	\author[C. Brittenham]{Charles Brittenham}
	\address{C. Brittenham, Department of Mathematics, 
		Colgate University,
		13 Oak Drive, Hamilton, NY 13346}
	\email{charles.brittenham@gmail.com}
	
	\author[J. Pakianathan]{Jonathan Pakianathan*}
	\address{J. Pakianathan, Department of Mathematics, 
		University of Rochester,
		500 Joseph C. Wilson Blvd., Rochester, NY 14627}
	\email{jonathan.pakianathan@rochester.edu}
	
	\date{}
	
	\thanks{*Corresponding Author: Email jonathan.pakianathan@rochester.edu}

	\begin{abstract}
	
	In this paper, we provide an application to the random distance-$t$ walk in finite planes and derive asymptotic formulas (as $q \to \infty$) for the probability of return to start point after $\ell$ steps	based on the "vertical" equidistribution of Kloosterman sums established by N. Katz. This work relies on a "Euclidean" association scheme studied in prior work of W.M.Kwok, E. Bannai, O. Shimabukuro and H. Tanaka.  We also provide a self-contained computation of the P-matrix and intersection numbers of this scheme for convenience in our application as well as a more explicit form for the intersection numbers in the planar case.	
		
		
	\end{abstract}
	
	\maketitle

	\setcounter{tocdepth}{1}
	\tableofcontents
	\setcounter{tocdepth}{3}

	\section{Introduction}
	
	In this paper we study the unit random walk in finite vector spaces. We also provide a discussion of a related association scheme that was the subject of the first author's thesis work~\cite{CB} that we latter were informed was previously studied by W.M. Kwok, E. Bannai, O. Shimabukuro and H. Tanaka in~\cite{BST}, \cite{Kwok}. In \cite{Kwok}, the P-matrix and intersection numbers of this scheme were worked out in terms of the character table of the "Euclidean" group over finite fields but had some errors which were latter corrected in \cite{BST} where a connection to Kloosterman sums was also mentioned, including Kloosterman's original bound. In this paper, we recap these calculations first to connect them more explicitly to our applications which use some deeper distributional data about Kloosterman sums provided by equidistribution results from number theory. (Along the way we also provide an explicit calculation of the planar intersection numbers implicitly described in~\cite{Kwok}.)
	
	The books~\cite{BH}, \cite{GC}, \cite{T} are good references  for the spectral graph theory concepts used in this paper. The book~\cite{BH} is a good reference for association schemes. 
	Let $\mathbb{F}_q$ be a finite field of odd prime power order $q$ and $V=\mathbb{F}_q^d$ be the standard $d$-dimensional vector space over $\mathbb{F}_q$.  Note that all finite vector spaces of odd characteristic are of this form.
	
	Equip $V$ with the standard bilinear dot product associated to the nondegenerate quadratic form $Q(v)=v_1^2 + \dots + v_d^2$. The "distance" between two points $v$ and $w$ in $V$ as determined by this quadratic form is given by $Q(v-w)$ which is the standard "Euclidean" distance formula without the square root. The quadratic space $(V,Q)$ will be refered to as $d$-dimensional Euclidean space over $\mathbb{F}_q$ in this paper. Note that it is known that there are only two nondegenerate quadratic forms up to isometry on any finite vector space (\cite{S}), besides the one derived from the standard dot product that we are using, the other one is a "Lorentzian" quadratic form with 
	$Q_{Lorentz}(v)=v_1^2 + \dots + v_{d-1}^2 + \xi v_d^2$ where $\xi$ is any fixed nonsquare element in $\mathbb{F}_q$. Most of what we discuss will also hold for this other quadratic form but we will stick to the dot product in this paper for brevity.
	
	For every $t \in \mathbb{F}_q$, we may define the distance-t graph whose vertex set is $V$ and where there is an edge between $v, w \in V$ if and only if $Q(v-w)=t$ and $v \neq w$. This regular graph has a $q^d \times q^d$ adjacency matrix $\mathbb{A}_t$ with respect to some fixed ordering of the vertices of $V$ and common vertex degree $|S_t|$ where $S_t=\{ v \in V\setminus\{0\} | Q(v)=t \}$ is the ''sphere of radius $t$'' and has order $q^{d-1}(1+o(1))$ when $t \neq 0$. (The exact order is known but here for brevity we collect secondary terms in $o(1)$ which tends to $0$ as $q \to \infty$). In fact this graph is a Cayley graph with connection set $S_t$ and hence is in fact vertex-transitive (i.e., there is a graph automorphism taking any vertex to any other vertex).
	The corresponding Markov chain (see \cite{MM} for basic Markov chain terminology) on this graph is called the distance-t random walk in $V$ with the case $t=1$ called the unit random walk in $V$. It corresponds to a situation where at each step, the current state evolves by taking a distance $t$ step with each such step equally likely. The transition matrix for this Markov chain is $T_t =\frac{1}{|S_t|}\mathbb{A}_t$.
	
	The matrices $\{ \mathbb{A}_t | t \in \mathbb{F}_q \}$ determine a symmetric association scheme with corresponding real and complex Bose-Messner algebras. What this means is that we have identities of the form $$\mathbb{A}_i \mathbb{A}_j = \sum_{k \in \mathbb{F}_q} p_{i,j}^k \mathbb{A}_k$$ where the intersection numbers $p_{i,j}^k$ correspond to the number of ways that a given ''line segment" of length $k$ can be completed to a triangle of side lengths $i,j$ and $k$. We call this particular scheme the Euclidean association scheme for the finite field $\mathbb{F}_q$ and recompute certain change of basis matrices $P$ and $Q$ (between a geometric and a spectral basis of the Bose-Messner algebra) which are generally important in the theory of such schemes. We also compute this association scheme's intersection numbers. (In~\cite{Kwok}, the higher dimensional numbers were computed in terms of the planar case which we compute explicitly here.) This first part of the paper's results appear in the first author's thesis as well as in~\cite{Kwok} (some errors in \cite{Kwok} were later corrected in \cite{BST} and the connection to Kloosterman sums made more explicit, though these were also already described earlier in~\cite{IR}). 
	
	The eigenvalues of $\mathbb{A}_t$ are related to twisted Kloosterman sums:
	$$
	\tilde{K}_d(a,b) = \sum_{x \in \mathbb{F}_q - \{0\}} \binom{x}{q}^d \chi(ax + \frac{b}{x})
	$$
	where $\binom{x}{q}$ is the Legendre symbol and $\chi$ is a fixed nontrivial additive character given by $\chi(x) = e^{\frac{2 \pi i Tr(x)}{p}}$, Here, $q=p^{\ell}$ with $p$ an odd prime and $Tr$ is the Galois trace from $\mathbb{F}_q$ to $\mathbb{F}_p$.
	
	In even dimensions $d$, the twisting by the Legendre symbol disappears and the eigenvalues are related to Kloosterman sums:
	$$
	K(a,b) = \sum_{x \in \mathbb{F}_q - \{0\}} \chi(ax + \frac{b}{x}).
	$$

	In general the $K(a,b)$ are real algebraic integers, contained in $\mathbb{Q}[e^{\frac{2 \pi i}{p}}]$,  the $p$th cyclotomic field.
	
	The identities $K(0,0)=q-1, K(a,b)=K(b,a), K(1,0)=-1$ are trivial to verify via character orthogonality and coordinate changes, as is the identity $K(a,b)=K(1,ab)$ when $a \neq 0$. By work of Kloosterman and others (\cite{K}, \cite{AW}), it is known $|K(1,\alpha)| \leq 2\sqrt{q}$ when $\alpha \neq 0$ and $q=p$ is an odd prime. One typically writes $K(1,\alpha)=2\sqrt{q} \cos(\theta_{\alpha,q})$ where $\theta_{\alpha,q} \in [0, \pi]$. These numbers occur as spectra of various natural Cayley graphs and hence in many combinatorial applications. (\cite{YD}, \cite{RB}, \cite{Hart})
	
	Nick Katz~\cite{NK} proved the deep result on "vertical equidistribution of Kloosterman sums" (\cite{XP1}) which states that for any $[a,b] \subseteq [0, \pi]$, the proportion of $\{\theta_{\alpha, q}, 1 \leq \alpha \leq q-1 \}$ that lie in $[a,b]$ approaches $ \frac{2}{\pi}\int_a^b sin^2(\theta) d\theta$ as $q \to \infty$. This result was obtained by estimating the $\ell$th moment of the Kloosterman numbers 
$M_{q,\ell}= \sum_{1 \leq \alpha \leq q-1} |K(1,\alpha)|^{\ell}$ for all positive integers $\ell$ reasonably. Useful closed form formulas for these moments are only known for a finite number of $\ell$ (\cite{XP2}).

	In this paper, we calculate $R_{q,\ell,t}$ (see Theorem~\ref{thm: Returnwalk}), the probability that you return to the vertex you started from after $\ell$-steps in the random distance $t$-walk, in terms of these moments of Kloosterman sums.  We show that this probability is independent of your starting state/vertex and is given by:
	\begin{theorem}[Probability of Return in the Distance $t$ Random Walk in $\mathbb{F}_q$-planes]
Let $q$ be an odd prime, $q=3 \text{ mod } 4$. Let $R_{q,\ell,t}$ be the probability that you return to the same vertex after $\ell$ steps in the distance-$t$ walk where $t \neq 0$.
Then $R_{q,\ell, t}=R_{q,\ell}$ is independent of $t \neq 0$ and initial state. We have 
$$
R_{q,\ell,t}=R_{q,\ell}=\frac{1}{q^2}(1 + \frac{(-1)^{\ell}}{(q+1)^{\ell-1}} M_{q,\ell})
$$
Furthermore, as $q \to \infty$ we have: 
$$
R_{q, 2\ell} = \frac{1}{q^2} + \frac{q^{\ell-1}}{(\ell+1)(q+1)^{2\ell-1}}\binom{2\ell}{\ell} (1 + o(1)) = \frac{1}{q^2} + \frac{1}{q^{\ell}(\ell+1)}\binom{2\ell}{\ell}(1+o(1))
$$
and
$$
R_{q, 2\ell+1}=\frac{1}{q^2}(1 - \frac{1}{(q+1)^{2\ell}}o(q^{\ell+1.5}))=\frac{1}{q^2}(1+o(q^{1.5-\ell}))
$$

\end{theorem}

	Note the first $\frac{1}{q^2}$ term in these asymptotic formulas is what one would expect if the location after $\ell$ steps were equally likely to be anywhere in the plane and the second term represents an arithmetic bias against that happening. In particular we see that for $\ell \geq 5$, we have $\mathbb{R}_{q,\ell}$ is $\frac{1}{q^2}(1+o(1))$ but for smaller $\ell$ arithmetic bias is significant and so we need at least $5$ steps to achieve uniformity. Note $R_{q,1}=0, R_{q,2}=\frac{1}{q}(1+o(1)), R_{q,4}=\frac{3}{q^2}(1+o(1))$ are all "arithmetically biased". Similar results hold for $q=1 \text{ mod } 4$ and can be seen in the paper. 
	
	Though the main terms in these formulas can be derived using just the spectral gap when $\ell \geq 5$, the explicit nature of the second order term requires, and is equivalent to, the deeper vertical equidistribution results.

	\section{Association scheme definitions}
	
	See \cite{BH} for a discussion of the essentials of the theory of association schemes. There are various equivalent definitions of association schemes, for us the following is most convenient:
	
	\begin{definition}[Association Scheme]
		An association scheme is a set $X$ equipped with a surjective "distance function" $d: X \to \Delta$ with distance set $\Delta$ which contains a formal zero element $\{ \bar{0} \}$ such that: \\
		(1) $d(x,y)=\bar{0}$ if and only if $x=y$.\\
		(2) $d(x,y)=d(y,x)$ all $x, y \in X$.\\
		and\\
		(3) Given $x,y \in X$, $k, i, j \in \Delta$ with $d(x,y)=k$, the number $p_{i,j}^k(x,y)$ of $z \in X$ such that $d(x,z)=i, d(z,y)=j$ only depends on $i, j, k$ and not $x, y$. Thus 
		$p_{i,j}^k(x,y)=p_{i,j}^k$ and these are called the "intersection numbers" of the association scheme.
	\end{definition}
	
	The best way to think of this last condition is to think of $x,y$ giving a segment of length $k$ and then noting that (3) states that the number of ways to complete this segment into a "triangle" only depends 
	on the side lengths of the triangle and not the endpoints of the segment itself.
	
	Though it is nice to think of an association scheme "geometrically" do note that we do not require the "distance" to satisfy the triangle inequality nor for the set $\Delta$ to be numerical. The elements of $\Delta$ can be any objects such as colors etc.
	
	For any $x \in X$ and $j \in \Delta$, the sphere of radius $j$ about $x$ is denoted $$S_j(x)=\{ y \in X | d(x,y)=j \}.$$ Note in any finite association scheme, setting $x=y$ in $(3)$ we see that $$p_{i,j}^{\bar{0}}=\delta_{i,j} |S_j(x)|$$ is independent of $x$. Here $\delta_{i,j}$ is the Kronecker Delta function which is $1$ when $i=j$ and $0$ otherwise. Furthermore by (2), we may reverse the role of $x$ and $y$ and so conclude that $$p_{i,j}^k = p_{j,i}^k$$ for all $i,j,k \in \Delta$.
	
	Given an association scheme $(X,d)$, we define the distance-$t$-graph for any $t \in \Delta$ as the graph on vertex set $X$ and where two vertices $x, y \in X$ are adjacent if and only if $d(x,y)=t$. 
	As $|S_t(x)|$ is independent of $x$, the distance-$t$-graph is $|S_t|$-regular.
	
	In the case $X$ is finite, we denote the adjacency matrix of the graph (with respect to some linear ordering of $X$) as $\mathbb{A}_t$. Thus $\mathbb{A}_{\bar{0}} = I$ and 
	$\sum_{j \in \Delta} \mathbb{A}_j = \mathbb{J}$ where $\mathbb{J}$ is the all $1$ matrix. It is easy to check that condition (2) is equivalent to the $\mathbb{A}_t$ being symmetric matrices and 
	condition (3) is equivalent to 
	$$
	\mathbb{A}_i \mathbb{A}_j = \sum_{k} p_{i,j}^k \mathbb{A}_k.
	$$
	As $p_{i,j}^k=p_{j,i}^k$, this says the matrices $\{ \mathbb{A}_i | i \in \Delta \}$ all commute with each other. Note as we require $d: X \times X \to \Delta$ to be surjective, no $\mathbb{A}_j$ is the zero matrix 
	and so the fact that they sum to $\mathbb{J}$ shows that the $\{ \mathbb{A}_i | i \in \Delta\}$ are a linearly independent set in $Mat_{|X|}(\mathbb{R})$. Condition (3), guarantees that their span is a algebra.
	
	\begin{definition}[Bose-Messner Algebra]
		Let $\mathfrak{A}=(X,d)$ be a finite association scheme with distance set $\Delta$ such that $|\Delta|=D+1$. We say $(X,d)$ has $D$ (nonzero) classes. The Bose-Messner algebra of the scheme is defined as
		$$BM(\mathfrak{A})=\mathbb{R}-\text{span of } \{ \mathbb{A}_j | j \in \Delta \} \subseteq Mat_{|X|}(\mathbb{R}).$$
		
		It is a $D+1$ dimensional $\mathbb{R}$-commutative algebra with basis $\{ \mathbb{A}_j | j \in \Delta \}$ and multiplication determined by the rule 
		$\mathbb{A}_i \mathbb{A}_j = \sum_{k} p_{i,j}^k \mathbb{A}_k$.
		
		The complexification of this algebra is called the complex Bose-Messner algebra and consists of the $\mathbb{C}$-span of the same $\mathbb{A}_j$ matrices.
	\end{definition}
	
	\begin{example}
		Let $X=(V,E)$ be a distance regular graph of diameter $D$, then $V$ equipped with the graph metric is an association scheme on $D$ (nonzero) classes. The theory of association schemes was introduced historically as a generalization of the theory of distance regular graphs.
	\end{example}
	
	The next example is the primary example considered in this paper.
	
	\begin{example}[Euclidean association scheme]
		Let $q$ be an odd prime power, $n\geq 2$, and $V=\mathbb{F}_q^n$ equipped with the "Euclidean" quadratic form $Q(v)=v_1^2+\dots+v_n^2$ and "distance" $Q(v-w)$.
		Then given $x \neq y,v \neq w \in V$ with $Q(x-y)=Q(v-w)$, Witt's theorem shows there is a isometry of $V$ (consisting of a composition of a translation and matrix multiplication by A where 
		$A \in O(n,q)=\{ \mathbb{A} \in Mat_n(\mathbb{F}_q) | A^T A = I \}$ is an "orthogonal" matrix), taking the pair $x,y$ to the pair of $v,w$. 
		
		When $n \geq 3$ or $n=2$ and $q=1 \text{ mod } 4$, it is possible $Q(x-y)=0$ but $x \neq y$ and so we distinguish between distance $0$ and distance $\bar{0}$ by declaring $d(x,y)=\bar{0}$ if and only if $x=y$ and $d(x,y)=Q(x-y)$ whenever $x \neq y$. Thus in $\mathbb{Z}_5^2$, $$d((0,0),(1,2))=1^2+2^2=5=0 \neq \bar{0}.$$ The distance set is $\Delta=\mathbb{F}_q \cup \{ \bar{0} \}$ when 
		$n \geq 3$ or $n=2, q=1 \text{ mod } 4$ but when $n=2, q=3 \text{ mod } 4$, $\Delta=(\mathbb{F}_q-\{ 0\})\cup \{ \bar{0} \}$ as $Q(x-y) \neq 0$ when $x \neq y$ in that case.
		
		Witt's theorem establishes that $(V,d)$ is an association scheme by establishing property (3).
	\end{example}
	
	\begin{example}
		Let $q$ be an odd prime power, $q = 3 \text{ mod 4}$. then $-1$ is not a square in $\mathbb{F}_q$ and so the Galois extension $E=\mathbb{F}_q[i]$ of $\mathbb{F}_q$ is a degree $2$ extension where $i$ is a root of $x^2+1$. The Galois norm map $N: E^{\times} \to \mathbb{F}_q^{\times}$ has $N(a+bi)=(a+bi)(a-bi)=a^2+b^2$ and this agrees with the Euclidean quadratic form $Q$ on $V=\mathbb{F}_q^2$ and so $Q(x-y)=0 \iff x=y$ in this case. Note as $N$ is a homorphism, $$|S_j((0,0))|=|N^{-1}(j)|=\frac{q^2-1}{q-1}=q+1$$ for all $j \in \mathbb{F}_q^{\times} = \mathbb{F}_q-\{ 0 \}.$ Thus in the corresponding association scheme, 
		$p_{j,j}^{\bar{0}} = q+1$ whenever $j \in \mathbb{F}_q^{\times}$ while $p_{\bar{0},\bar{0}}^{\bar{0}}=1$.
	\end{example}
	
	Given a finite association scheme $\mathfrak{A}=(X,d)$ on $D$ classes, the matrices $\{ \mathbb{A}_j | j \in \Delta \}$ are a real basis for the corresponding Bose-Messner algebra $BM(\mathfrak{A})$.
	As the matrices are symmetric and as they commute, they can be simultaneously (orthogonally) diagonalized on their action on $\mathbb{R}^X$ and this vector space splits as an orthogonal direct sum of simultaneous eigenspaces of all the $\mathbb{A}_j$ operators. These simultaneous eigenspaces are called weight spaces. Let $\{ E_i, i \in \Delta' \}$ be the collection of orthogonal projection operators to the individual weight spaces. Linear algebra guarantees that each $E_i$ is a polynomial expression in the $A_j$'s and hence lies in the Bose-Messner algebra and thus is a linear combination of the $A_j$. Conversely as each $A_j$ is constant on the image of any $E_i$, each $A_j$ can be trivially written as a linear combination of the $E_i$.
	Thus the $\{ E_i | i \in \Delta' \}$ are a different, "spectral basis" of the Bose Messner algebra $BM(\mathfrak{A})$ and hence $|\Delta| = |\Delta'|$. We will refer to the original basis $\{ \mathbb{A}_j | j \in \Delta \}$ of $BM(\mathfrak{A})$
	as the geometric basis of the Bose-Messner algebra as these encode the distances of the scheme.
	
	Many of the main results of association schemes, and distance regular graphs in particular, arise from the interactions between the geometric and spectral basis of the Bose-Messner algebra of the scheme.
	
	In particular two fundamental change of basis matrices are defined as follows:
	
	\begin{definition}
		Let $(X,d)$ be a finite association scheme on $D$ (nonzero) classes and let $\{ A_j | j \in \Delta \}$ and $\{ E_i | i \in \Delta'\}$ denote the  "geometric" and "spectral" basis of the corresponding Bose-Messner algebra.
		The $(D+1) \times (D+1)$ real matrices $P$ and $Q$ are defined via
		$$
		A_j = \sum_{i \in \Delta'} P_{i,j} E_i \text{ for all } j \in \Delta
		$$
		and 
		$$
		E_i = \frac{1}{|X|} \sum_{j \in \Delta} Q_{j,i} A_j \text{ for all } i \in \Delta'
		$$
		Thus $P Q = |X| I = Q P$.
	\end{definition}
	
	Note the $\{ P_{i,j} | i \in \Delta'\}$ are just the eigenvalues of $A_j$, the adjacency matrix of the distance $j$-graph of the scheme. Note that $|\Delta'|=|\Delta|=D+1$ is usually much smaller than $|X|$ 
	and so the $|X| \times |X|$ matrix $A_j$ has lots of multiplicities in its eigenvalues, encoded by $rank(E_i)$ which is the dimension of the $i$th weight space.
	
	\begin{definition}[Intersection Matrices]
		Given an association scheme on $D$ (nonzero) classes with intersection numbers $p_{i,j}^k$, $i,j,k \in \Delta$. The $(D+1) \times (D+1)$ intersection matrix $L_i$ is defined via $(L_i)_{k,j}=p_{i,j}^k$ for all $i \in \Delta$.
		
		It turns out the eigenvalues of $L_j$ are the same as the eigenvalues of the $|X| \times |X|$ matrix $A_j$, i.e. are also $\{ P_{i,j} | i \in \Delta'\}$.
	\end{definition}
	
	The matrices $P, Q, L_j, A_i, E_j$ satisfy many interesting identities in an association scheme. For our purpose we focus on Delsarte's linear programming bound.
	Let $(X,d)$ be a finite association scheme and $Y \subseteq X$. For each $j \in \Delta$, define $a_j=\frac{| \{ (y_1,y_2) \in Y \times Y | d(y_1,y_2)=j \}|}{|Y|} =\frac{\chi_Y^T A_j \chi_Y}{\chi_Y^ T \chi_Y} \in \mathbb{Q}_{\geq 0}$ where $\chi_Y$ is the characteristic column vector of $Y$. Then the row vector $a=(a_j)_{j \in \Delta}$ is called the inner distribution of $Y$ as it encodes the frequency of various distances on the pairs in the subset $Y$. This inner distribution satisfies Delsarte's linear programming bound condition
	$a Q \geq 0$ as well as $\sum_{j \in \Delta} a_j = |Y|$ and these conditions give us many constraints on distances induced on subsets $Y$ of the scheme. 
	
	\section{Eigenvalues of the Distance $t$-graphs}
		
	For any $t \in \mathbb{F}_q$, let $X(q,t)$ denote the associated distance $t$-graph on $\mathbb{F}_q^d$. In this section, we describe the spectrum of the adjacency matrix of this graph, $A_t$, when $t \neq \bar{0}$, i.e., for $t \in \mathbb{F}_q$. This spectrum was previously computed by various authors for example in \cite{IR}, \cite{T} but we include a brief discussion here to be self-contained.

	Standard spectral theory of finite Cayley graphs over Abelian groups, shows that the complex eigenfunctions of this graph's adjacency operator $A_t$ are the characters $\chi_m(x)=\chi(m \cdot x)$ as $m$ ranges over $\mathbb{F}_q^d$. Here $\cdot$ stands for dot product. The corresponding eigenvalue of $A_t$ on $\chi_m$ which we will denote $\lambda_{m,t,d}$ (recall $d$ is the dimension of $V$) is given by 
	$$
	\lambda_{m,t,d} =\sum_{x \in S_{t,d}} \chi_m(x) = \sum_{x \neq 0, Q(x)=t} \chi(m \cdot x).
	$$
	
	Clearly $\lambda_{0,t,d} = |S_{t,d}|$. The size of these spheres were  originally calculated by Minkowski, and there are at most 3 distinct sphere sizes corresponding to $t=0$, $t$ a nonsquare in $\mathbb{F}_q$, and $t$ a nonzero square in $\mathbb{F}_q$. This is because scaling by $\lambda \in \mathbb{F}_q^{\times}$ gives a bijection between $S_{t,d}$ and $S_{\lambda^2 t, d}$.
	
	When $m,t$ are both nonzero, these eigenvalues are given by twisted Kloosterman sums. We provide the calculation for the convenience of the reader. First note that by character orthogonality we have $\frac{1}{q}\sum_{s \in \mathbb{F}_q} \chi(sL)=\delta_{L,0}$ where $\delta_{L,0}=1$ if $L=0$ and $0$ otherwise. Then we compute, when $t \neq 0$,:

	\begin{align*}
		\lambda_{m,t,d} &=  \sum_{\{x | Q(x)=t\}} \chi(m \cdot x) =\frac{1}{q} \sum_{s \in \mathbb{F}_q} \sum_{x \in V} \chi(m \cdot x) \chi(s(Q(x)-t)) \\
		&= \frac{1}{q}\sum_{s \in \mathbb{F}_q} \chi(-st) \sum_{x_1,\dots, x_d \in \mathbb{F}_q} \chi(m_1x_1 + \dots + m_dx_d + s(x_1^2 + \dots + x_d^2)) \\
		&=\frac{1}{q}\sum_{s \in \mathbb{F}_q} \chi(-st) \prod_{i=1}^d \sum_{x_i \in \mathbb{F}_q} \chi(m_i x_i + sx_i^2) \\
		&=q^{d-1} \delta_{m,0} + \frac{1}{q}\sum_{s \neq 0} \chi(-st) \prod_{i=1}^d \sum_{x_i \in \mathbb{F}_q} \chi(m_i x_i + sx_i^2) \\
		&=q^{d-1} \delta_{m,0} + \frac{1}{q}\sum_{s \neq 0} \chi(-st) \prod_{i=1}^d \sum_{y_i \in \mathbb{F}_q} \chi(sy_i^2-\frac{m_i^2}{4s}) \\
	\end{align*}
	where in the last step we completed the square and substituted $y_i=x_i+\frac{m_i}{2s}$. Letting $G(s)=\sum_{y \in \mathbb{F}_q} \chi(sy^2)$ denote the corresponding Gauss sum, we conclude that 
	$$
	\lambda_{m,t,d}=q^{d-1} \delta_{m,0} + \frac{1}{q}\sum_{s \neq 0} \chi(-st-\frac{Q(m)}{4s}) G(s)^d.
	$$
	When $t=0$ the only correction needed is to exclude $x=0$ from the set $Q(x)=0$ which results in reducing this quantity by $1$. Thus we have generally that:
	$$
	\lambda_{m,t,d}=q^{d-1} \delta_{m,0} -\delta_{t,0} + \frac{1}{q}\sum_{s \neq 0} \chi(-st-\frac{Q(m)}{4s}) G(s)^d.
	$$
	
	It is a well known result of Gauss that $G(s)=\binom{s}{q} \epsilon(q) \sqrt{q}$ where $\binom{s}{q}$ is the Legendre symbol which is $+1$ when $s$ is a nonzero square modulo $q$, and $-1$ when $s$ is a nonsquare modulo $q$. The quantity $\epsilon(q)=1$ when $q \equiv 1 \text{ mod } 4$ and $\epsilon(q)=i$ when $q \equiv 3 \text{ mod } 4$. Thus we conclude:
	
	$$
	\lambda_{m,t,d}=q^{d-1} \delta_{m,0} -\delta_{t,0} + q^{\frac{d}{2}-1} \epsilon(q)^d \tilde{K}_d(-t,-\frac{Q(m)}{4})
	$$
	
	where $\tilde{K}_d(a,b)=\sum_{s \neq 0} \binom{s}{q}^d \chi(as+\frac{b}{s})$ is a twisted Kloosterman sum. Note that the dependence of $\tilde{K}_d(a,b)$ on $d$ is weak and only depends on the parity of $d$. When $d$ is even, $\tilde{K}_d(a,b)=K(a,b)=\sum_{s \neq 0}  \chi(as+\frac{b}{s})$ is the regular Kloosterman sum. Note that the $m$-dependence of the eigenvalue $\lambda_{m,t,d}$ 
	is only through $Q(m)=||m||=m_1^2+\dots+m_d^2$ and whether $m$ is the origin or not. Following the convention adopted in association schemes, we will write $Q(m)=\bar{0}$ if $m=0$ and  $Q(m)=0$ if $m \neq 0$ but $m_1^2+\dots+m_d^2=0$. 
	
	Thus $Span\{\chi_m | Q(m)=k\}=L_k$ is contained in an eigenspace of $A_t$ for each $t,k \in \mathbb{F}_q \cup \{ \bar{0} \}$. The complex vector space $\mathbb{C}[V]$, of all complex functions on $V$, decomposes as a orthogonal direct sum of the $L_k$ with respect to the standard Hermitian inner product. The $L_k$ are simultaneous eigenspaces of all the $A_t$ operators and $L_0$ is $0$ if and only if $S_0 = \emptyset$. Thus the number of nonzero $L_k$ is the same as the dimension of the corresponding Euclidean Association scheme and so the nonzero $L_k$ are exactly the weight spaces in the decomposition of the complex Bose-Messner algebra of the scheme. Note that $L_{\bar{0}}=Span\{\chi_0\} =Span\{ \mathfrak{1} \}$ consists of the constant functions.
	
	We record these results in the following theorem:
	
	\begin{theorem}
		\label{thm: Kloostermansareeigenvalues}
		Let $q$ be an odd prime power and $\mathbb{F}_q$ the finite field of order $q$. For $d$ an integer $\geq 2$, $m \in \mathbb{F}_q^d, t \in \mathbb{F}_q$ we have that the eigenvalue 
		$\lambda_{m,t,d}$ of the adjacency matrix $A_t$ of the distance-$t$-graph of $V=\mathbb{F}_q^d$ on the character $\chi_m(x)=\chi(m \cdot x)$ is given by:
		
		$$
		\lambda_{m,t,d}=q^{d-1} \delta_{m,0} -\delta_{t,0} + q^{\frac{d}{2}-1} \epsilon(q)^d \tilde{K}_d(-t,-\frac{Q(m)}{4})
		$$
		where $\tilde{K}_d(a,b)=\sum_{s \neq 0} \binom{s}{q}^d \chi(as+\frac{b}{s})$ is a twisted Kloosterman sum.
		
		If $W_k=Span\{\chi_m | Q(m)=k\}$ in $\mathbb{C}[V]$ then $\mathbb{C}[V]=\oplus_{k \in \mathbb{F}_q \cup \{ \bar{0} \}} W_k$ is an orthogonal direct sum decomposition of the vector space of complex valued functions on $V$ into simultaneous eigenspaces of all the $A_t$ operators, $t \in \mathbb{F}_q \cup \{ \bar{0}\}$. The orthogonal projection $E_k$ onto the (nonzero) $W_k$ give the spectral basis of the complex Bose Messner Algebra of the corresponding Euclidean association scheme.		
	\end{theorem}
	
	Note it is not hard to see that the complex conjugate of $\tilde{K}_d(a,b)$ is $$\sum_{s \neq 0} \binom{s}{q}^d \chi(-as - \frac{b}{s})=\sum_{s \neq 0} \binom{-s}{q}^d \chi(as + \frac{b}{s})$$ 
	and so $\tilde{K}_d(a,b)$ is real if either $d$ is even or $q\equiv 1 \text{ mod } 4$ (i.e. when $-1$ is a square modulo $q$) but is purely imaginary when $d$ is odd and $q \equiv 3 \text{ mod 4}$. However the quantity $\epsilon(q)^d \tilde{K}_d(a,b)$ and hence $\lambda_{m,t,d}$ are real numbers of course as they are eigenvalues of symmetric matrices.
	
	\section{Some properties of (twisted) Kloosterman sums}
	
	Let $q=p^{\ell}$ be an odd prime power and define $\tilde{K}_d(a,b)=\sum_{s \neq 0} \binom{s}{q}^d \chi(as+\frac{b}{s})$ for $a,b \in \mathbb{F}_q$.
	
	Let $\xi_p=e^{\frac{2 \pi i}{p}}$ be a primitive $p$th root of unity and $\mathbb{Q}[\xi_p]$ be the corresponding cyclotomic field. Note as $\chi(x)=\xi^{Tr(x)}$ where $Tr: \mathbb{F}_q \to \mathbb{F}_p$ is the Galois trace, we have that $\tilde{K}_d(a,b)$ are elements in the ring of integers of $\mathbb{Q}[\xi_p]$. Recall that the Galois group $Gal(\mathbb{Q}[\xi_p]/\mathbb{Q})$ is Abelian of order $p-1$ and contains automorphisms of the form $\psi_a$ determined by the property that they take $\xi$ to $\xi^a$ for all $a \in \mathbb{F}_p^*$ and fix the rational subfield.
	
	\begin{proposition}
	\label{prop:basicKloostermanIdentities}
		Let $p$ be an odd prime and $q=p^{\ell}$ and define $$\tilde{K}_d(a,b)=\sum_{s \neq 0} \binom{s}{q}^d \chi(as+\frac{b}{s})$$ for $a,b \in \mathbb{F}_q$. These are algebraic integers in the cyclotomic field $\mathbb{Q}[\xi_p]$ which satisfy: \\
		(1) $\tilde{K}_d(a,b)=\tilde{K}_d(b,a)$ \\
		(2) $\tilde{K}_d(a,b)=\binom{a}{q}^d\tilde{K}_d(1,ab)$ when $a \neq 0$ \\
		(3) The quantity $\tilde{K}_d(0,0)$ is equal to $q-1$ when $d$ is even and equal to $0$ when $d$ is odd. Thus $\tilde{K}_d(0,0)=(q-1)\frac{(-1)^d+1}{2}$. \\
		(4) The quantity $\tilde{K}_d(1,0)=\tilde{K}_d(0,1)$ equals $-1$ when $d$ is even and $G(1)$ when $d$ is odd where $G(1)=\sum_{s \in \mathbb{F}_q} \chi(s^2)=\epsilon(q)\sqrt{q}$ is a quadratic Gauss sum.\\
		(5) For $c \in \mathbb{F}_p^*$ we have that $\psi_c(\tilde{K}_d(a,b))=\tilde{K}_d(ca,cb)$ where $\psi_c$ is the Galois automorphism of $\mathbb{Q}[\xi_p]$ over $\mathbb{Q}$ which takes $\xi_p$ to $\xi_p^c$.
		
	\end{proposition}
	\begin{proof}
		To prove (1), use the change of variable $s \to \frac{1}{s}$. To prove (2) use the change of variable $s'=as$ for nonzero $a$. (3) follows trivially once one notes that in any finite field, an equal number of elements are nonzero squares as are nonsquares.
		
		For (4), first note that when $d$ is even we have 
		$\tilde{K}_d(1,0)=\sum_{s \neq 0}  \chi(s)$ which by character orthogonality is equal to $-\chi(0)=-1$. When $d$ is odd, we have $\tilde{K}_d(1,0)=\sum_{s \neq 0}  \binom{s}{q}\chi(s)=\sum_{s \in \mathbb{F}_q}  \binom{s}{q}\chi(s)=\sum_{s \in \mathbb{F}_q}  (\binom{s}{q}+1)\chi(s)=\sum_{u \in \mathbb{F}_q} \chi(u^2)=G(1).$
		
		Finally for (5), we note that $\tilde{K}_d(a,b)=\sum_{s \neq 0} \binom{s}{q}^d \xi_p^{Tr(as+\frac{b}{s})}$ and apply the Galois automorphism $\psi_c$ to get 
		$$\psi_c(\tilde{K}_d(a,b)) = \sum_{s \neq 0} \binom{s}{q}^d \xi_p^{cTr(as+\frac{b}{s})}=\sum_{s \neq 0} \binom{s}{q}^d \xi_p^{Tr(cas+\frac{cb}{s})}=\tilde{K}_d(ca,cb)$$ 
		as the Galois trace $Tr: \mathbb{F}_q \to \mathbb{F}_p$ is $\mathbb{F}_p$-linear.
		
	\end{proof}
	
	Note that when $d$ is even, the $\tilde{K}_d(a,b)$ coincide with (untwisted) Kloosterman sums $K(a,b)=\sum_{s \neq 0} \chi(as+\frac{b}{s})$ and the properties above apply to those also. 
	In particular when $q=p$, this implies that the $\{ K(1,u) | u \text{ a nonzero square } \}$ are all Galois conjugates in $\mathbb{Q}[\xi_p]$. This is because for $c \neq 0$, $\psi_c(K(1,1))=K(c,c)=K(1,c^2)$.
	
	We will need one more property of these Kloosterman numbers that requires some Fourier analysis. Recall if $f: V \to \mathbb{C}$ then we define the Fourier transform 
	$$\hat{f}(m)=\frac{1}{q^d}\sum_{x \in V} f(x) \chi(-m \cdot x).$$ It is easy to show via character orthogonality that we can recover $f$ via $$f(x) = \sum_{m \in V} \hat{f}(m)\chi(m \cdot x).$$
	
	Note if we let $S_t$ denote the indicator function of the corresponding $t$-sphere, $t \in \Delta$, then $q^d \hat{S}_t(m)=\sum_{x \in \mathbb{F}_q^d} S_t(x)\chi(-m \cdot x) = \sum_{x \in S_t} \chi(m \cdot x) = \lambda_{m,t,d}$ from the previous section. (We define $\lambda_{m,\bar{0},d}=1$ to also make it work in that case). By disjointness of spheres of different radii, it follows that for any $i \neq j \in \Delta$, 
	
	\begin{align*} 0 = S_i(x)S_j(x) &=(\sum_{m \in V} \hat{S}_i(m)\chi(m \cdot x) )(\sum_{n \in V} \hat{S}_j(n) \chi(n \cdot x)) \\
	0 &=\sum_{t \in V} (\sum_{m \in V} \hat{S}_i(m)\hat{S}_j(t-m)) \chi(t \cdot x). \\
	\end{align*}
	
By character independence, we conclude that $\sum_{m \in V} \hat{S}_i(m) \hat{S}_j(t-m) = 0$ for all $t \in \mathbb{F}_q^d$ when $i \neq j \in \Delta$. When $i=j$, as $S_i(x)S_i(x)=S_i(x)$, we instead conclude 
$\hat{S}_i(t) = \sum_{m \in V} \hat{S}_i(m) \hat{S}_i(t-m)$.

Using that $q^d \hat{S}_t(m)=\lambda_{m,t,d}$ and that $\hat{S}_t(m)=\hat{S}_t(-m)$, this becomes the next proposition:

\begin{proposition}
\label{pro: sillyprop}
Consider the Euclidean association scheme on $V=\mathbb{F}_q^d$ with distance set $\Delta$. Let $\lambda_{m,t,d}$ be as in Theorem~\ref{thm: Kloostermansareeigenvalues} with $\lambda_{m,\bar{0},d}=1$. Then for any $i \neq j \in \Delta$ and $t \in V$, we have 
$$
\sum_{m \in V} \lambda_{m,i,d} \lambda_{t-m,j,d} = 0,
$$
and in particular when $t=0$,
$$
\sum_{m \in V} \lambda_{m,i,d} \lambda_{m,j,d} = 0.
$$
When $i=j$ we instead have:
$$
\sum_{m \in V} \lambda_{m,i,d} \lambda_{t-m,i,d} = q^d \lambda_{t,i,d},
$$
and
$$
\sum_{m \in V} \lambda_{m,i,d}^2 = q^d \lambda_{0,i,d}=q^{2d}\hat{S}_i(0)=q^d|S_i|.
$$

\end{proposition}
	
\section{The P and Q matrices of the Euclidean Association scheme}

         The calculation of the $P$-matrix of the Euclidean scheme in this section was first done in \cite{Kwok} with corrections in \cite{BST} but we provide the details independently here in language more connected to our applications.

	Let $\Delta$ be the "distance" set of the Euclidean association scheme on $V=\mathbb{F}_q^d$. Thus if $d \geq 3$ or $(d=2, q=1 \text{ mod } 4)$ we have $\Delta=\{ \bar{0} \} \cup \mathbb{F}_q$ 
	whereas when $(d=2, q=3 \text{ mod } 4)$, we have $\Delta=\{ \bar{0} \} \cup \mathbb{F}_q^{\times}$.
	
	By Theorem~\ref{thm: Kloostermansareeigenvalues}, it follows that both the spectral and geometric basis of the complex Bose-Messner algebra can be indexed by the distance set $\Delta$.
	Fix a linear ordering of $\Delta$ where $\bar{0}$ comes first and if $0 \in \Delta$, it comes second. For example the ordering $\bar{0} < 0 < 1 < \dots < q-1$ works when $(d=2, q=1 \text{ mod } 4)$ or $d \geq 3$ and when otherwise, i.e., when $(d=2, q=3 \text{ mod } 4)$ we drop $0$ from the list. 
	
	Then when $i, j \in \Delta$, we have $\mathbb{P}_{i,j}$ is by definition the eigenvalue of $A_j$ on the $i$th weight space 
	$W_i$ which was denoted $\lambda_{m,j,d}$ in previous sections where $m \in V$ is any vector with $Q(m)=i$. In other words, by definitions it follows that $\lambda_{m,j,d}=\mathbb{P}_{Q(m),j}$ for all $m \in V, j \in \Delta$.
	
	In particular, $\mathbb{P}_{\bar{0},t} = \lambda_{0,t,d}=q^d \hat{S}_t(0)=|S_t|$ for all $t \in \Delta$ and thus the top row of $\mathbb{P}$ lists the sphere sizes.
	Also $\mathbb{P}_{i,\bar{0}} = \lambda_{m,\bar{0},d}=1$ (here $m$ is any vector with $Q(m)=i$). Thus the leftmost column of $\mathbb{P}$ is all 1's (eigenvalues of $\mathbb{A}_{\bar{0}}$ which is the identity operator).
	
	On the other hand, If $0 \in \Delta$, $\mathbb{P}_{0,t} = \lambda_{m,t,d}$ for some nonzero vector $m$ of length $Q(m)=0$. By Theorem~\ref{thm: Kloostermansareeigenvalues}, 
	$\mathbb{P}_{0,t}=-\delta_{t,0} + q^{\frac{d}{2}-1}\epsilon(q)^d \tilde{K}_d(-t,0)=\mathbb{P}_{\bar{0},t}-q^{d-1}=|S_t|-q^{d-1}$. Thus the $0$-row (which would be listed second if it occurs) consists 
	of the deviations of sphere sizes from the expected size $q^{d-1}$.
	
	Proposition \ref{pro: sillyprop} can be recast in terms of the entries of $\mathbb{P}$ as follows:
	When $i \neq j \in \Delta$, we have $\sum_{m \in V} \lambda_{m,i,d} \lambda_{m,j,d} = 0$. Lumping terms in this sum by the length $t$ of the vector $m$ yields:
	$$
	\sum_{t \in \Delta} |S_t| \mathbb{P}_{t,i} \mathbb{P}_{t,j} = 0	
	$$
	Similarly we get when $i = j \in \Delta$, $\sum_{t \in \Delta} |S_t| \mathbb{P}_{t,i}^2 = q^d |S_i|$.
	
	It follows that $\mathbb{P}^T D \mathbb{P} = q^d D$ where $D$ is a diagonal matrix with the sizes of the spheres along the diagonal in the same order as they appear in the top row of $\mathbb{P}$.
	
	As no sphere sizes are zero (we only index distances that occur), $D$ is invertible and it follows that $\mathbb{P}^{-1} = \frac{1}{q^d} D^{-1} \mathbb{P}^T D$.
	
	As the $Q$-matrix is defined uniquely (as inverses are unique) by the condition $PQ=q^d I = QP$, it follows that $Q=D^{-1} \mathbb{P}^T D$. We summarize this next:
	
	\begin{theorem}
	\label{thm:PQmatrixmain}
	Let $q$ be any odd prime power, $d \geq 2$ and consider the Euclidean association scheme with distance set $\Delta$. Then $|\Delta|=q+1$ when $0 \neq \bar{0} \in \Delta$ and $|\Delta|=q$ otherwise when $0 \notin \Delta$, i.e. when $(d=2$ and $q = 3 \text{ mod } 4)$.
	
	The $|\Delta| \times |\Delta|$ matrices $\mathbb{P}$ and $Q$ of the association scheme (relative to some linear ordering of $\Delta$ for which $\bar{0}$ comes first) have:
	$\mathbb{P}_{\bar{0},j} = |S_j|$ for all $j \in \Delta$, $\mathbb{P}_{i,\bar{0}} = 1$ for all $i \in \Delta$ and otherwise for $i, j \in \Delta-\{ \bar{0} \}$ we have 
	$$
	P_{i,j} = -\delta_{j,0} + q^{\frac{d}{2}-1} \epsilon(q)^d \tilde{K}_d(-j,-\frac{i}{4})
         $$
	were $\tilde{K}_d(a,b)=\sum_{x \in \mathbb{F}_q^{\times}} \binom{x}{q}^d \chi(ax+\frac{b}{x})$ is a twisted Kloosterman sum and $\epsilon(q)$ is either $1$ or $i$ depending if $q$ is $1$ or $3$ mod $4$.
	
	If $D$ is the diagonal matrix whose diagonals are the sphere sizes (following the chosen linear ordering of $\Delta$), then $Q=D^{-1}\mathbb{P}^T D$ and $Q\mathbb{P}=\mathbb{P}Q=q^d\mathbb{I}$. Thus $Q_{i,j}=\frac{|S_j|}{|S_i|} P_{j,i}$.
	\end{theorem}
	
	Note that when $d=2, q=3 \text{ mod } 4$, then $0 \notin \Delta$ and $P_{i,j}=-K(-j,\frac{-i}{4})=-K(1,\frac{ij}{4})$ for any $i, j \in \Delta-\{ \bar{0} \}$ and so other than the first row and column of $\mathbb{P}$, the other entries are negatives of Kloosterman numbers. In this case, we also have $|S_i|=q+1$ for all $i \neq \bar{0}$ and $|S_{\bar{0}}|=1$. If we denote $\mathfrak{1}$ to be the $q-1$ dimensional column vector of all $1$s, and $K$ to the symmetric $(q-1) \times (q-1)$ matrix with $K_{i,j}=K(1,\frac{ij}{4})=K_{j,i}, i,j \in \mathbb{F}_q^{\times}$ then we have:
$$\mathbb{P}=\begin{bmatrix} 1 & (q+1)\mathfrak{1}^T \\ \mathfrak{1} & -K   \end{bmatrix}=Q.$$
Thus $\mathbb{P}^2=q^2\mathbb{I}$ for the Euclidean Association scheme when $d=2, q = 3 \text{ mod } 4$

On the other hand, when $d=2, q=1 \text{ mod } 4$, then $0 \in \Delta$, $|S_i|=q-1$ if $i \neq 0, \bar{0}$ and $|S_0|=2q-2, |S_{\bar{0}}|=1$. With respect to any linear ordering with $\bar{0} < 0 < \mathbb{F}_q^{\times}$, the $(q+1) \times (q+1)$ $\mathbb{P}$ and $Q$ matrices are given in this case by:
$$\mathbb{P}=\begin{bmatrix} 1 & 2q-2 & (q-1)\mathfrak{1}^T \\ 1 & q-2 & -\mathfrak{1}^T \\ \mathfrak{1} & (-2)\mathfrak{1} & K   \end{bmatrix}$$ and
again one can verify that $\mathbb{P}=Q$ using the formulas of the last theorem.
Thus again one has $\mathbb{P}^2=q^2\mathbb{I}$ though the dimensions of $\mathbb{P}$ are  one larger in this case compared to the last one. 

We record these dimension 2 results in the next proposition:

\begin{proposition}
Let $\mathbb{F}_q$ be a finite field of odd characteristic and consider the Euclidean Association scheme when dimension $d=2$. Let $\mathfrak{1}$ to be the $q-1$ dimensional column vector of all $1$s, and $K$ be the symmetric $(q-1) \times (q-1)$ matrix with $K_{i,j}=K(1,\frac{ij}{4})=K_{j,i}, i,j \in \mathbb{F}_q^{\times}$ (with respect to some linear ordering of $\mathbb{F}_q^{\times}$). Order the distance set $\Delta$ via a linear ordering with $\bar{0} < 0 < \mathbb{F}_q^{\times}$.

Then if $q=3 \text{ mod } 4$, we have $$\mathbb{P}=\begin{bmatrix} 1 & (q+1)\mathfrak{1}^T \\ \mathfrak{1} & -K   \end{bmatrix}=Q$$
are $q \times q$ matrices.

On the other hand, when $q=1 \text{ mod } 4$, we have
$$\mathbb{P}=\begin{bmatrix} 1 & 2q-2 & (q-1)\mathfrak{1}^T \\ 1 & q-2 & -\mathfrak{1}^T \\ \mathfrak{1} & (-2)\mathfrak{1} & K   \end{bmatrix}=Q$$ are
 $(q+1) \times (q+1)$ matrices.

In either case, $\mathbb{P}^2=q^2\mathbb{I}$.

\end{proposition}

The case when $d \geq 3$ is a bit trickier as the dimension of spheres can vary more. In general, for any $\lambda \in \mathbb{F}_q^{\times}$, scaling by $\lambda$ gives a bijection between the sphere of 
radius $j$ about the origin and the sphere of radius $\lambda^2 j$ about the origin. Thus there are in general 4 distinct sizes of spheres $|S_{\hat{0}}|=1$, $|S_0|$, $|S_{Sq}|$ and $|S_{Sq^c}|$ 
where $|S_{Sq}|$ is the common size of $S_j$ when $j$ is a nonzero square in $\mathbb{F}_q$ and $|S_{Sq^c}|$ is the common size of $S_j$ when $j$ is a nonsquare in $\mathbb{F}_q$.
As these spheres partition $\mathbb{F}_q^d$, we have $$q^d=1+|S_0| + \frac{q-1}{2}|S_{Sq}| + \frac{q-1}{2} |S_{Sq^c}|$$ for all integers $d \geq 3$ and odd prime powers $q$. 

Pick a primitive generator $\theta$ of the multiplicative group of $\mathbb{F}_q$ (which is a cyclic group of order $q-1$). Then $\theta^j$ is a square in $\mathbb{F}_q$ if and only if $j$ is even and so the Legendre symbol $\binom{\theta^j}{q}$ is equal to $(-1)^j$ for all integers $j$. We order the distance set $\Delta=\{ \bar{0} \} \cup \mathbb{F}_q$ via $\bar{0} < 0 < \theta^0 < \theta^1 < \theta^2 < \dots < \theta^{q-2}$. Note $4=\theta^s$ for unique integer $s$ with $0 \leq s \leq q-2$.

We define the $(q-1) \times (q-1)$ matrix $\tilde{\mathbb{K}}_d$ (whose rows and columns are indexed by $\{0,1,\dots, q-2\}$ in the natural order) via

\begin{align*}
(\tilde{\mathbb{K}}_d)_{i,j} 
&= \epsilon(q)^d \tilde{K}_d(-\theta^j, -\frac{\theta^i}{4}) \\
&=\epsilon(q)^d \binom{-\theta^j}{q}^d \tilde{K}_d(1,\frac{\theta^i\theta^j}{4})\\
&=\epsilon(q)^d\binom{-1}{q}^d (-1)^{jd}\tilde{K}_d(1,\theta^{i+j-s})\\
\end{align*}

where $\tilde{K}(a,b)$ are the (twisted) Kloosterman sums. The second equality follows from part (2) of Proposition~\ref{prop:basicKloostermanIdentities}. 

This matrix is always a real matrix from the earlier discussion about twisted Kloosterman sums. When $d$ is even, it is symmetric and circulant. Recall, a circulant matrix is one  such that each row of the matrix is the cyclic shift of the row above it, one notch to the left. When $d$ is odd, it is no longer circulant or symmetric but is signed-circulant, i.e., each row is the negative of the cyclic shift of the row above it, one notch to the left.

Using these conventions and Theorem~\ref{thm:PQmatrixmain}, the following proposition follows by a simple computation with the help of the identities in Proposition~\ref{prop:basicKloostermanIdentities}: 

\begin{proposition}
Let $\mathbb{F}_q$ be a finite field of odd characteristic and consider the Euclidean Association scheme of dimension $d \geq 3$. Fix a primitive generator $\theta$ of $\mathbb{F}_q^{\times}$, then $4=\theta^s$ for unique integer $0 \leq s \leq q-2$.

Let $(\tilde{\mathbb{K}}_d)_{i,j} =\epsilon(q)^d\binom{-1}{q}^d (-1)^{jd}\tilde{K}_d(1,\theta^{i+j-s})$ whose rows and columns are indexed by $\{0,\dots, q-2\}$. It is a real, symmetric and circulant matrix when $d$ is even and it is a real, signed-circulant matrix when $d$ is odd.

Let $\mathfrak{1}$ be the $(q-1)$ dimensional column vector of all ones and let 
$\hat{\eta}$ tbe the $(q-1)$ dimensional row vector whose entries alternate between 
$|S_{Sq}|$ and $|S_{Sq^c}|$, starting with the former. Let $\hat{\mu}$ be the $(q-1)$ dimensional column vector whose entries are $(-1)^{id}$ as $i$ ranges from $0$ to $q-2$. (So it always starts at $1$ and alternates sign if $d$ is odd, but is the all one vector when $d$ is even).

With respect to the ordering $\bar{0} < 0 < \theta^0 < \theta^1 < \theta^2 < \dots < \theta^{q-2}$ of the distance set $\Delta$, the $P$ matrix of the Euclidean association scheme is a $(q+1)\times(q+1)$ matrix given by: 
$$\mathbb{P}=\begin{bmatrix} 1 & |S_0| & \hat{\eta} \\ 1 & \alpha & \beta\hat{\mu}^T \\ \mathfrak{1} & \beta\hat{\mu}-\mathfrak{1} & q^{\frac{d}{2}-1}\tilde{\mathbb{K}}_d   \end{bmatrix}$$
where $\alpha=-1+q^{\frac{d}{2}-1}\epsilon(q)^d (q-1)\frac{(-1)^d+1}{2}$ and 
$\beta=-q^{\frac{d}{2}-1}\epsilon(q)^d$ when $d$ even, $\beta=\binom{-1}{q} q^{\frac{d-1}{2}}\epsilon(q)^{d+1}$ when $d$ odd. 

The $Q$ matrix has $Q_{i,j}=\frac{|S_j|}{|S_i|}P_{j,i}$. Due to varying sizes of spheres when $d \geq 3$, one no longer has $Q=P$ in general.

\end{proposition}

The above proposition also works for the $d=2$ case though in the case $q=3 \text{ mod } 4$ one should throw out the 2nd row and column as the distance $0$ does not occur.

\section{Equidistribution}

Through work of Kloosterman and A. Weil, for any odd prime power $q$ and $a \in \mathbb{F}_q^{\times}$, it is known that the Kloosterman numbers 
$K_q(1,a)=\sum_{x \in \mathbb{F}_q^{\times}} \chi(x + \frac{a}{x})$ satisfy
$$
K_q(1,a)=\sqrt{q}(e^{i \theta_{q,a}} + e^{-i \theta_{q,a}})=2\sqrt{q}\cos(\theta_{q,a})
$$
for a unique "Kloosterman" angle $\theta_{q,a} \in [0, \pi]$.

Using the sophisticated tools of lisse sheafs and etale cohomology, N. Katz (slightly simplified proofs later by Adolphson), proved the "vertical equidistribution" of these numbers that states that as $q \to \infty$, the distribution of these angles approaches the Sato-Tate measure on $[0,\pi]$. We discuss this more carefully next.

\begin{definition}[Sato-Tate measure]
The Sato-Tate Borel probability measure $\mu_{ST}$ on $[0,\pi]$ is given by the condition
$$
\mu_{ST}([a,b]) = \frac{2}{\pi} \int_a^b sin^2(\theta) d\theta
$$
for all $0\leq a < b \leq \pi$.

$\mu_{ST}$ is characterized as the unique Borel measure being absolutely continuous with respect to Lebesgue measure with Radon-Nikodym derivative $\frac{2}{\pi} sin^2(\theta)$. 

For any continuous $f: [0, \pi] \to \mathbb{R}$, we write $$
E_{ST}[f] = \frac{2}{\pi} \int_0^{\pi} f(\theta) sin^2(\theta) d\theta$$
for the expectation of $f$ with respect to this probability measure.
\end{definition}

\begin{definition}[Kloosterman angle average]
For any odd prime power $q$, let $\theta_{a,q}$ be the Kloosterman angles associated to the Kloosterman sums $K(1,a), a \in \mathbb{F}_q^{\times}$. Given a continuous function 
$f: [0, \pi] \to \mathbb{R}$, we let 
$$
E_{K,q}[f] = \frac{1}{q-1} \sum_{a \in \mathbb{F}_q^{\times}} f(\theta_{q,a})
$$
be the "sample average" of $f$ over these Kloosterman angles.
\end{definition}

We are now ready to state the deep vertical equidistribution theorem of N. Katz (\cite{NK}):

\begin{theorem}[Vertical Equidistribution of Kloosterman sums]
For any sequence of odd prime powers $q_n \to \infty$, and any continuous function $f: [0, \pi] \to \mathbb{R}$, we have 
$$
\lim_{n \to \infty} E_{K,q_n}[f] = E_{ST}[f].
$$
\end{theorem}

Next we will discuss some graph theoretical equivalents to the equidistribution theorem. First some basic trigonometric facts will be collected:

\begin{proposition}
\label{pro:trig}
For any nonnegative integers $m, n$ we have: \\
(1) $\frac{2}{\pi} \int_{0}^{\pi} \cos(m \theta)\cos(n \theta)d\theta = \delta_{m,n}+\delta_{m,n}\delta_{m,0}.$\\
(2) $E_{ST}[\cos(n\theta)]=0$ if $n \neq 0,2$. $E_{ST}[\cos(2\theta)]=-\frac{1}{2}$. \\
The Sato-Tate probability measure is characterized by these expectations amongst continuous probability measures on $[0, \pi]$. \\
(3) For every $\ell \geq 1$ we have $$2^{2\ell-1}\cos^{2\ell}(\theta)=\sum_{k=0}^{\ell-1} \binom{2\ell}{k} \cos((2\ell-2k)\theta) + \frac{1}{2} \binom{2\ell}{\ell}.$$
(4) For every $\ell \geq 1$ we have $$2^{2\ell}\cos^{2\ell+1}(\theta)=\sum_{k=0}^{\ell} \binom{2\ell+1}{k} \cos((2\ell+1-2k)\theta).$$ \\
(5) For positive integer $m$ we have $E_{ST}[\cos^{2m+1}(\theta)]=0$ and 
$$E_{ST}[cos^{2m}(\theta)]=\frac{1}{2^{2m-1}} (\frac{-1}{2} \binom{2m}{m-1} + \frac{1}{2} \binom{2m}{m})=\frac{1}{2^{2m}(m+1)}\binom{2m}{m}.$$
\end{proposition}
\begin{proof}
Proof of (1): Follows from the trigonometric identity
$$
\cos(mx)\cos(nx) = \frac{1}{2}\cos((m+n)x) + \frac{1}{2}\cos((m-n)x)
$$
and simple integration.

Proof of (2): Follows from (1), once one notes that $\sin^2(\theta)=\frac{1-\cos(2\theta)}{2}$. The characterization of Sato-Tate measure from these expectations follows from the density of the algebra generated by the $cos(n\theta)$ in the ring of continuous real valued functions on $[0,\pi]$ (with the supremum metric) which itself follows from a generalized Stone-Weierstrass theorem, together with the Riesz representation theorem. \\

Proof of (3): Write $2\cos(\theta)=e^{i\theta} + e^{-i\theta}$ and raise both sides to the $2\ell$ power to conclude 
$$
2^{2\ell} \cos^{2\ell}(\theta)=\sum_{k=0}^{2\ell} \binom{2\ell}{k} (e^{i\theta})^{2\ell-k} (e^{-i\theta})^k = \sum_{k=0}^{2\ell} \binom{2\ell}{k} (e^{i\theta})^{2\ell-2k}.
$$
Finish by grouping the $k=j$ and $k=2\ell-j$ terms for each $j$.

Proof of (4): Follows the same procedure as for (3).

Proof of (5): Follows from using (3),(4) in (2).
\end{proof}

\begin{definition}[Kloosterman moments]
For any odd prime power $q$, and positive integer $\ell$, let 
$$
M_{q,\ell} = \sum_{a \in \mathbb{F}_q^{\times}} K_q(1,a)^\ell = 2^{\ell} q^{\frac{\ell}{2}}(q-1) E_{K,q}[cos^{\ell}(\theta)]
$$
be the $\ell$th Kloosterman moment.
\end{definition}

Using little-oh notation and the trigonometric identities in Proposition~\ref{pro:trig}, we find that the vertical equidistribution of Kloosterman sums is equivalent to establishing the following behavior for Kloosterman moments:

$$
M_{q, 2\ell+1} = o(q^{\ell+1.5}) 
$$
and
$$
M_{q,2\ell} =  q^{\ell+1} \frac{1}{\ell+1}\binom{2\ell}{\ell} (1 + o(1))
$$
as $q \to \infty$.

The quantity $C_{\ell}=\frac{1}{\ell+1}\binom{2\ell}{\ell}$ in the last limit is the $\ell$th Catalan number which occurs frequently in combinatorics. It is for example the number of ways to bracket a given $(\ell+1)$-fold product in terms of pairwise multiplications.

Recall, in a graph, two edges are incident if they share a common vertex and a walk of length $\ell$ is a sequence of $\ell$ edges, where each edge is incident to the previous one. If $\mathbb{A}$ is the adjacency matrix of the graph with respect to some vertex ordering then the $(v,w)$-entry of $\mathbb{A}^{\ell}$ is the number of walks of length $\ell$ from vertex $v$ to vertex $w$ in the graph. Such a walk is called closed if $v=w$. Thus $Trace(\mathbb{A}^{\ell})$ is the number of closed walks of length $\ell$ in the graph. 

Let $q$ be an odd prime power, $q= 3 \text{ mod } 4$. Fix any $t \in \mathbb{F}_q^{\times}$ then we have seen that the spectrum of $\mathbb{A}_t$, the adjacency matrix of the distance-t-graph for the plane $\mathbb{F}_q^2$ is 
the multiset $$\{\{ (q+1)^{(1)}, -K_q(1,a)^{(q+1)} | a \in \mathbb{F}_q^{\times} \} \}$$ 
where the superscripts indicate the multiplicity of each eigenvalue in the multiset.

Thus
$$
Trace(\mathbb{A}_t^{\ell}) = (q+1)^{\ell} + (q+1)(-1)^{\ell}\sum_{a \in \mathbb{F}_q^{\times}}
 K_q(1,a)^{\ell} = (q+1)^{\ell} + (q+1)(-1)^{\ell}M_{q,\ell}.
$$

On the other hand, when $q=1 \text{ mod } 4$, the spectrum of $\mathbb{A}_t$ is 
$$\{\{ (q-1)^{(1)}, (-1)^{(2q-2)}, K_q(1,a)^{(q-1)} | a \in \mathbb{F}_q^{\times} \} \}$$ 
and so 
$$
Trace(\mathbb{A}_t^{\ell}) = (q-1)^{\ell} + (2q-2)(-1)^{\ell} + (q-1)\sum_{a \in \mathbb{F}_q^{\times}}
 K_q(1,a)^{\ell} = (q-1)^{\ell} + (2q-2)(-1)^{\ell} + (q-1)M_{q,\ell}.
$$

We record these results in the next proposition:

\begin{proposition}
\label{pro:closedwalks}
Let $q$ be an odd prime power and $t \in \mathbb{F}_q^{\times}$ and $\mathbb{A}_t$ be the adjacency matrix of the distance $t$ graph for the plane $\mathbb{F}_q^2$. 
Then if $q = 3 \text{ mod } 4$, we have the number of closed walks of length $\ell$ is:
$$
Trace(\mathbb{A}_t^{\ell}) = (q+1)^{\ell} + (q+1)(-1)^{\ell}M_{q,\ell}.
$$
If $q=1 \text{ mod } 4$ we have:
$$
Trace(\mathbb{A}_t^{\ell})  = (q-1)^{\ell} + (2q-2)(-1)^{\ell} + (q-1)M_{q,\ell}
$$
where $M_{q,\ell}$ is the $\ell$th Kloosterman moment.

Vertical equidistribution of Kloosterman angles is equivalent to
$$
M_{q, 2\ell+1} = o(q^{\ell+1.5}) 
$$
and
$$
M_{q,2\ell} =  q^{\ell+1} \frac{1}{\ell+1}\binom{2\ell}{\ell} (1 + o(1))
$$
as $q \to \infty$.

\end{proposition}

\begin{corollary}
For any odd prime power $q$, we have 
$M_{q,1}=1$ and $M_{q,2}=q^2-q-1$. 
Thus the equidistribution conditions hold for these moments.
\end{corollary}
\label{cor: KloosterMoments}
\begin{proof}
We prove the $q=3 \text{ mod } 4$ case with similar calculations working for the $q=1 \text{ mod } 4$ case. Setting $t=1$ in Proposition~\ref{pro:closedwalks}, we see $Trace(\mathbb{A}_1)=0$ as there are no closed walks of length one. Thus $(q+1)-(q+1)M_{q,1}=0$ giving 
$M_{q,1}=1$. A closed walk of length two consists of walking out using an edge and returning using the same edge and so the number of those is $q^2(q+1)$ as there are $q^2$ choices of initial vertices and $q+1$ many adjacent vertices to walk to.
Thus $q^2(q+1)=Trace(\mathbb{A}_1^2)=(q+1)^2 + (q+1)M_{q,2}$ yielding the stated formula for $M_{q,2}$. 
\end{proof}

\section{An application to the random walk in finite planes of odd order}

Fix $t \neq 0$ in $\mathbb{F}_q$ and let $\mathbb{A}_t$ be the adjacency matrix for the distance $t$-graph in the plane $\mathbb{F}_q^2$. Define $T_t=\frac{1}{|S_t|} \mathbb{A}_t$ to be the transition matrix of the corresponding random walk Markov chain. 
The $(i,j)$-entry of $T_t^{\ell}$ represents the probability of transitioning from the $i$th vertex to the $j$th vertex after $\ell$ steps in this Markov chain. As the $A_t$-graph is vertex transitive, the diagonal entries of $T_t^{\ell}$ are all equal as the probability that we return to vertex $i$ after $\ell$ steps, if we start at vertex $i$, will be independent of the vertex. Thus $(T_t^{\ell})_{i,i}=\frac{1}{q^2}Trace(T_t^{\ell})$ for any $i$. 

Using Bayesian conditioning, it follows that for any initial probability distribution $p_i$ on the vertices, the chance that you start and end at the same vertex after $\ell$ steps is 
$ \sum_{j=1}^{q^2} (T_t^{\ell})_{j,j}p_j= \frac{1}{q^2} Trace (T_t^{\ell})\sum_{j=1}^{q^2} p_j =\frac{1}{q^2|S_t|^{\ell}} Trace(A_t^{\ell})$ is the probability that you return to where you start after $\ell$ steps (it is independent of initiate state). Let us call this quantity the probability of return and denote it $R_{q,\ell,t}$. Now $|S_t|=q \pm 1$ where the sign depends if $q$ is $1$ or $3$ mod $4$ so we find that this probability is $R_{q,\ell,t}=\frac{1}{q^2(q \pm 1)^{\ell}} Trace(A_t^{\ell})$.

By Corollary~\ref{cor: KloosterMoments}, we see that if $q=3 \text{ mod } 4$, and positive $\ell$ we have 
$$
R_{q,\ell,t}=R_{q,\ell}=\frac{1}{q^2}(1 + \frac{(-1)^{\ell}}{(q+1)^{\ell-1}} M_{q,\ell})
$$
does not depend on $t \neq 0$. Note the first $\frac{1}{q^2}$ term is what one would expect if the location after $\ell$ steps were equally likely to be anywhere in the plane and the second term represents an arithmetic bias against that happening. Vertical equidistribution of Kloosterman sums, tells us furthermore that 
$$
R_{q, 2\ell} = \frac{1}{q^2} + \frac{q^{\ell-1}}{(\ell+1)(q+1)^{2\ell-1}}\binom{2\ell}{\ell} (1 + o(1)) = \frac{1}{q^2} + \frac{1}{q^{\ell}(\ell+1)}\binom{2\ell}{\ell}(1+o(1))
$$
and
$$
R_{q, 2\ell+1}=\frac{1}{q^2}(1 - \frac{1}{(q+1)^{2\ell}}o(q^{\ell+1.5}))=\frac{1}{q^2}(1+o(q^{1.5-\ell}))
$$
as $q \to \infty$ thru such prime powers.

Thus we have proven:
\begin{theorem}[Probability of Return in the Distance $t$ Random Walk in $\mathbb{F}_q$-planes]
\label{thm: Returnwalk}
Let $q$ be an odd prime, $q=3 \text{ mod } 4$. Let $R_{q,\ell,t}$ be the probability that you return to the same vertex after $\ell$ steps in the distance-$t$ walk where $t \neq 0$.
Then $R_{q,\ell, t}=R_{q,\ell}$ is independent of $t \neq 0$ and initial state. We have 
$$
R_{q,\ell,t}=R_{q,\ell}=\frac{1}{q^2}(1 + \frac{(-1)^{\ell}}{(q+1)^{\ell-1}} M_{q,\ell})
$$
Furthermore, as $q \to \infty$ we have: 
$$
R_{q, 2\ell} = \frac{1}{q^2} + \frac{q^{\ell-1}}{(\ell+1)(q+1)^{2\ell-1}}\binom{2\ell}{\ell} (1 + o(1)) = \frac{1}{q^2} + \frac{1}{q^{\ell}(\ell+1)}\binom{2\ell}{\ell}(1+o(1))
$$
and
$$
R_{q, 2\ell+1}=\frac{1}{q^2}(1 - \frac{1}{(q+1)^{2\ell}}o(q^{\ell+1.5}))=\frac{1}{q^2}(1+o(q^{1.5-\ell}))
$$

\end{theorem}

\section{Intersection Numbers and Matrices of the scheme}
\label{section: intersectionnumbers}

In this section we calculate the intersection numbers $p_{i,j}^k$ of the $d$-dimensional Euclidean Association scheme. Recall $p_{i,j}^k=p_{j,i}^k$ is the number of ways a pair of points $\{x,y\}$ in $\mathbb{F}_q^d$ with $d(x,y)=k$ can be completed to a triangle $\{x,y,z\}$ of side lengths $i,j,k$. Note $p_{i,j}^{\bar{0}} =\delta_{i,j} |S_i|$ and 
$p_{\bar{0},j}^k =\delta_{j,k}$ so we may assume $i,j,k \in \mathbb{F}_q$ and so are not $\bar{0}$.

We will concentrate on the planar, $d=2$ case as the higher dimensional cases were reduced to this case in \cite{Kwok}. This case was partially computed in \cite{TP} in a different context.

\begin{theorem}[Intersection numbers of planar Euclidean association scheme]
\label{thm:intersection numbers}
Let $q$ be an odd prime power and $p_{i,j}^k$ be the intersection numbers of the planar ($d=2$) Euclidean association scheme for $i,j,k \in \mathbb{F}_q$ not equal to $\bar{0}$.

If $q=1 \text{ mod } 4$ we have $p_{i,j}^0=1+(q-2)\delta_{i,j}\delta_{i,0}-\delta_{i,j}$.

Otherwise $k \neq 0$ and we have
$$
p_{i,j}^k=\binom{4\sigma_2-\sigma_1^2}{q}+1
$$
where $\binom{x}{q}$ is the usual Legendre symbol, $\sigma_1=i+j+k, \sigma_2=ij+jk+ki$.

Furthermore when $i,j,k \in \mathbb{F}_q$ are the distances in a triple of points of $\mathbb{F}_q^2$, we have $4\sigma_2 = \sigma_1^2$ if and only if $\{i,j,k\}$ are the distances of a collinear triple of points (a triple of points that lies on an affine line).
\end{theorem}
\begin{proof}
WLOG we may take $x=(0,0)$ and Witt's theorem guarantees $p_{i,j}^k$ will be independent of $y=(u_1,v_1) \neq (0,0)$ but only depend on its length $u_1^2+v_1^2=k$. We may furthermore assume $u_1 \neq 0$. 

Then by definiton, any $z=(u_2,v_2)$ making $\{x,y,z\}$ a $i-j-k$ triangle must satisfy $u_2^2+v_2^2=i, (u_2-u_1)^2+(v_2-v_1)^2=j$. Given $u_1, v_1, i, j, k$, we need to count the number, $p_{i,j}^k$, of $z=(u_2,v_2)$ that solve these equations. Plugging the first two equations into the last one, we get 
$u_1u_2+v_1v_2 = \frac{k+i-j}{2}$. The set of $(u_2,v_2)$ solving this equation is an affine line (not necessarily thru the origin) perpendicular to the line thru $y=(u_1,v_1)$. It is not hard to see that in fact $(u_2,v_2)=(\frac{k+i-j}{2u_1},0) + s(v_1,-u_1)$ for some $s \in \mathbb{F}_q$. However we still need $u_2^2+v_2^2=i$ which yields 
$(\frac{k+i-j}{2u_1}+sv_1)^2 + s^2u_1^2=i$. This yields the equation 
$ks^2 +(k+i-j)\frac{v_1}{u_1}s +((\frac{k+i-j}{2u_1})^2 -i)=0$ which is a quadratic equation unless $k=0$.

When $k=0$, $v_1^2=-u_1^2 \neq 0$ so the resultant linear equation has a unique solution for $s$ whenever $j \neq i$ and $p_{i,j}^0=1$ for all $i \neq j \in \mathbb{F}_q$. 
When $k=0$ and $i=j$, the equation cannot hold unless $i=0$ also in which case any $s$ works as long as $x \neq y$ and $x \neq z$ so there are $(q-2)$ such $s$ (In this case as $k=0$ the original line thru $y=(u_1,v_1)$ is its own perpendicular and so one must avoid the two choices of $s$ where $z$ coincides with $x=(0,0)$ or $y$). Thus $p_{i,i}^0=(q-2)\delta_{i,0}$. Note this $k=0$ case can occur only when $q=1 \text{ mod } 4$.

Otherwise $k \neq 0$ and we get a quadratic equation for $s$ whose discriminant can be calculated to be $ (k+i-j)^2(\frac{v_1}{u_1})^2-4k((\frac{k+i-j}{2u_1})^2 -i)=\frac{1}{u_1^2}((k+i-j)^2v_1^2 -k(k+i-j)^2+4kiu_1^2)$. Using $v_1^2=k-u_1^2$ this becomes 
$4ki-(k+i-j)^2=2(ki+ij+jk)-(k^2+i^2+j^2)=4\sigma_2-\sigma_1^2$.

Thus there are $\binom{4\sigma_2-\sigma_1^2}{q}+1$ solutions for $s$ where $\binom{x}{q}$ is the usual Legendre symbol.

For the last statement, let 
$B=\begin{bmatrix} u_1 & u_2 \\ v_1 & v_2 \end{bmatrix}$ and note that $B^TB=\begin{bmatrix} k & \frac{k+i-j}{2} \\ \frac{k+i-j}{2} & i \end{bmatrix}$ by the calculations above. The triple of points $\{ (0,0), (u_1,v_1), (u_2,v_2) \}$ is 
collinear if and only if $rank(B) < 2$ if and only if $det(B)=0$ if and only if 
$det( \begin{bmatrix} k & \frac{k+i-j}{2} \\ \frac{k+i-j}{2} & i \end{bmatrix})=\frac{4\sigma_2-\sigma_1^2}{4}=0$.
\end{proof}

We record the following corollary, which exploits the odd behavior of isotropic lines when they exist:

\begin{corollary}
Let $q$ be an odd prime power with $q=1 \text{ mod } 4$. Let 
$E \subseteq \mathbb{F}_q^2$ have $|E| > q$ and $\Delta'(E)$ be the nonzero distances achieved by $E$. Then $|E| \leq q + |\Delta'(E)|(|\Delta'(E)|-1)$.
\end{corollary}
\begin{proof}
As $|E| > q$, there exist two distinct points $x, y \in E$ with $d(x,y)=0$. Let $E'=E \backslash L$ where $L$ is the line through $x$ and $y$. This line $L$ is isotropic in the sense that distances between distinct points on the line are always zero. Note that $|E'| \geq |E|-q$. 

By Theorem~\ref{thm:intersection numbers}, $p_{i,j}^0=1$ for all $i,j \in \mathbb{F}_q^{\times}$ with $i \neq j$ and $p_{i,i}^0=0$ unless $i=0$. Furthermore, by the collinearity condition $4\sigma_2-\sigma_1^2=0$, it is easy to check that a point $z \neq x, y$ is collinear with $x$ and $y$ if and only if $d(z,x)=d(z,y)=0$ also.

It follows that each point $z$ off the isotropic line $L$ thru $x$ and $y$ is uniquely determined in the plane by the pair of unequal, nonzero distances $i=d(z,x)$ and $j=d(z,y)$. Thus $|E|-q \leq |E'| \leq |\Delta'(E)|(|\Delta'(E)|-1)$ from which the corollary follows.
\end{proof}

We now recall the definition of the intersection matrices of an association scheme.

\begin{definition}[Intersection Matrices]
Let $(V,d)$ be an association scheme with distance set $\Delta$ and intersection numbers $p_{i,j}^k$. We define $|\Delta|$ many $|\Delta| \times |\Delta|$ intersection matrices of the scheme $L_i, i \in \Delta$ via $$(L_i)_{k,j} = p_{i,j}^k.$$

In general, these satisfy $L_{\bar{0}}=\mathbb{I}$ and $L_i L_j = \sum_{k \in \Delta} p_{i,j}^k L_k$. Thus the map $$\mathbb{R}[A_i, i \in \Delta] \to \mathbb{R}[L_i, i \in \Delta]$$ is an algebra epimorphism from the Bose-Messner algebra to the real algebra generated by these intersection matrices. Note the dimensions $|V| \times |V|$ of the $A_i$ matrices are in general quite different than the size of the $L_i$ matrices.

The eigenvalues of the $L_i$ matrix are the same as that of the $A_i$ matrix but with different multiplicities. In fact the columns of the scheme's $Q$ matrix  are the corresponding simultaneous (right) eigenvectors of the $L_i$ while the rows of the $P$ matrix are the (left) eigenvectors of the $L_i$. These facts follow from the identity 
$$PL_jP^{-1}=diag(P_{\bar{0},j},\dots, P_{d,j}).$$ (See section 11.2 of \cite{BH}.)
\end{definition}                                                                                                                                                                                                                                                                                                                                                                           

We now record the intersection matrices for the planar Euclidean association scheme:

\begin{theorem}
Let $q$ be an odd prime power and consider the planar Euclidean association scheme on $\mathbb{F}_q^2$. Note $L_{\bar{0}}$ is always an identity matrix. Recall when $q=1 \text{ mod } 4$, we have $\bar{0} \neq 0 \in \Delta$ also.

When $i,j \in \mathbb{F}_q^{\times}$, we have 
$(L_i)_{\bar{0},j}=p_{i,j}^{\bar{0}}=\delta_{i,j}|S_i|=\delta_{i,j}(q-\epsilon_q^2)$ and 
$(L_i)_{0,j}=p_{i,j}^{0}=1-\delta_{i,j}$.

We have $(L_i)_{j,\bar{0}}=p_{i,\bar{0}}^j=\delta_{i,j}$ and $(L_i)_{\bar{0},\bar{0}}=p^{\bar{0}}_{i,\bar{0}}=0$. 

When $i,j \in \mathbb{F}_q$,$k \in \mathbb{F}_q^{\times}$ we have 
$$
(L_i)_{k,j}=p^k_{i,j}=\binom{4\sigma_2 - \sigma_1^2}{q} + 1 \in \{0,1,2\}
$$
where $\sigma_2=ij+jk+ki, \sigma_1=i+j+k$ and $\binom{x}{q}$ is the Legendre symbol.

When $q=3 \text{ mod } 4$, the spectrum of $L_i$ is $\{\{ (q+1)^{(1)}, -K_q(1,a)^{(1)}, a \in \mathbb{F}_q^{\times} \}\}$ for any $i \in \mathbb{F}_q^{\times}$. Thus $Trace(L_i^{\ell})=(q+1)^{\ell} + (-1)^{\ell}M_{q,\ell}$.

When $q=1 \text{ mod } 4$, the spectrum of $L_i$ is $\{\{ (q-1)^{(1)}, -1^{(1)}, K_q(1,a)^{(1)}, a \in \mathbb{F}_q^{\times} \}\}$ for any $i \in \mathbb{F}_q^{\times}$. Thus $Trace(L_i^{\ell})=(q-1)^{\ell} + (-1)^{\ell} + M_{q,\ell}$.

\end{theorem}
              
Note that the entries of $L_i$ keep track of which triangles exist in the plane $\mathbb{F}_q^2$ and through the last theorem, implicitly determine the Kloosterman moments $M_{q,\ell}$ for all $\ell$. As it is well known that a finite multiset of complex numbers is determined by all its moments, this means the Kloosterman sums as a set are determined by exactly the data of which triangles exist in the plane $\mathbb{F}_q^2$.

\end{document}